\documentclass{elsarticle}
\usepackage{natbib}

\usepackage{hyperref,multirow}
\journal{Journal of Computational and Applied Mathematics }

\usepackage{graphicx,subeqn,bbold,amssymb }
\usepackage{algorithm}
\usepackage{algpseudocode}
\newtheorem{theorem}{Theorem}

\newproof{proof}{Proof}
\begin{document}

\begin{frontmatter}

\title{Computing Symmetric Positive Definite Solutions of Three Types of Nonlinear Matrix Equations}

\author[mymainaddress]{Negin Bagherpour}
\author[mymainaddress]{Nezam Mahdavi-Amiri\corref{mycorrespondingauthor}}
\cortext[mycorrespondingauthor]{Corresponding author}
\ead{nezamm@sharif.ir}

\address[mymainaddress]{Department of Mathematical Sciences, Sharif University of Technology, Tehran, Iran}

\begin{abstract}
Nonlinear matrix equations arise in many practical contexts related to control theory, dynamical programming and finite element methods for solving some partial differential equations. In most of these applications, it is needed to compute a symmetric and positive definite solution. Here, we propose new iterative algorithms for solving three different types of nonlinear matrix equations. We have recently proposed a new algorithm for solving positive definite total least squares problems. Making use of an iterative process for inverse of a matrix, we convert the nonlinear matrix equation to an iterative linear one, and, in every iteration, we apply our algorithm for solving a positive definite total least squares problem to solve the linear subproblem and update the newly defined variables and the matrix inverse terms using appropriate formulas. Our proposed algorithms have a number of useful features. One is that the computed unknown matrix remains symmetric and positive definite in all iterations. As the second useful feature, numerical test results show that in most cases our proposed approach turns to compute solutions with smaller errors within lower computing times. Finally, we provide some test results showing that our proposed algorithm converges to a symmetric and positive definite solution in Matlab software environment on a PC, while other methods fail to do so.

\end{abstract}


\end{frontmatter}

%\linenumbers

\newcommand{\tr}{\mathop{\mathrm{tr}}}
\section{Introduction}
In several practical applications concerned with solving partial differential equations, control theory and ladder network design a symmetric and positive definite solution of a nonlinear matrix equation needs to be computed; e.g., see \cite{2,5,6}. We consider the nonlinear matrix equation
\begin{equation}\label{19}
X+{\sum}_{i=1}^mA_i^Tf_i(X)A_i=Q,
\end{equation}
with $A_i$, $i=1,\cdots,m,$ and $Q$ being $n\times n$ matrices. For some special set of functions $f_i$, $i=1,\cdots,m,$ equation (\ref{19}) turns to the usual nonlinear equations. Sayed \cite{19} considered a class of nonlinear equations of the general form (\ref{19}) with some special choices of the $f_i$. He introduced an iterative algorithm to solve the equations. Here, we discuss three cases of interest:\\

\textbf{Case 1: $m=1, f_1(X)=X^{-1}.$} This leads to
\begin{equation}\label{20}
X+A^TX^{-1}A=Q,
\end{equation}
which arises in contexts related to control theory; e.g., see \cite{8}. Zhou \cite{8} discussed a method for solving (\ref{20}). A similar equation
\begin{equation}\label{21}
X+A^{\ast}X^{-q}A=Q
\end{equation}
is also concidered, where $0<q\leq 1$; e.g., see \cite{9,21}. In 2005, Hasanov \cite{9} introduced a method for solving (\ref{21}). Also, in 2013, Yin \cite{10} outlined a novel method for solving (\ref{21}). Assuming $A=B+iC$ and $Q=M+iN$, a complex form of (\ref{21}) is defined, whose solution was discussed by Guo \cite{13}.\\

\textbf{Case 2: $m=1, f_1(X)=-X^{-2}.$} Then, the nonlinear equation 
\begin{equation}\label{22}
X-A^TX^{-2}A=Q
\end{equation}
is at hand. This equation arises in solving special types of partial differential equations using finite element methods; e.g., see \cite{5,17}. Ivanov \cite{17} disscused two iterative methods for solving (\ref{22}). Cheng \cite{11} derived a purterbation analysis of the Hermitian solution to (\ref{22}). Also, Sayed \cite{14} and Ivanov \citep{5} introduced iterative methods for solving slightly different equations, $X-A^TX^{-n}A=Q,
$ with $n\geq 2$ being an integer, and $X+A^TX^{-2}A=Q$.\\

\textbf{Case 3: $m=2, f_1(X)=X^{-t_1}, f2_(X)=X^{-t_2}.$} This results in the equation
\begin{equation}\label{23}
X^s+A_1^TX^{-t_1}A_1+A_2^TX^{-t_2}A_2=Q
\end{equation}
with different applications in control theory, dynamic programming and statistics;  e.g., see \cite{6,12}. In 2010, Liu \cite{6} described a method for solving (\ref{23}). Also, Long \cite{12} considered an special case of (\ref{23}) with $t_1=t_2=1$ . After discussing some properties of symmetric and positive definite solutions of the corresponding nonlinear equation, Long outlined an iterative method to solve it. Pei \cite{20} and Duan \cite{18} considered another special case with $A_2=0$. They studied the conditions for existence of a symmetric and positive definite solution to the corresponding nonlinear equation and outlined two different algorithms to compute it.\\ Solving some other nonlinear matrix equations has also been discussed in the litrature. For instance, solutions of the nonlinear equations  
\begin{center}
$X={\sum}_{l=0}^{k-1}P_l^TX^{{\alpha}_l}P_l$
\end{center}
and 
\begin{center}
$X={\sum}_{l=0}^{k-1}{P_l^TXP_l}^{{\alpha}_l}$
\end{center}
have been considered in \cite{15,16}. 

The remainder of our work is organized as follows. In Section 2, we describe our general idea for computing a symmetric positive definite solution to the above three different types of nonlinear matrix equations. In Section 3, we provide the details and outline three algorithms for solving the equations. Computational results and comparisons with available methods are given in Section 4. Section 5 gives our concluding remarks.

\section{The general Idea}
In \cite{1}, we recently proposed a new method for solving a positive definite total least squares problem.
There, the goal was to compute a symmetric and positive definite solution of the over-determined linear system of equations
\begin{equation}\label{1}
DX \simeq T,
\end{equation}
where $D,T \in {\mathbb{R}}^{m \times n},$ with $m \geq n,$ are known, using a total error formulation. Unlike the ordinary 
least squares formulation, in a total formulation both $D$ and $T$ are assumed to contain error. Hence, we proposed an error function
\begin{equation}\label{2}
f(X)=\tr{(DX-T)}^T(D-TX^{-1}),
\end{equation}
with $\tr(\cdot)$ standing for the trace of a matrix. Then, the solution of the positive definite total least squares problem (\ref{1}) was considered to be the symmetric and positive definite matrix $X$ minimizing $f(X)$. In \cite{1}, we proposed positive definite total least squares with Cholesky decomposition algorithm (PDTLS-Chol) and positive definite total least squares with spectral decomposition algorithm (PDTLS-Spec) to compute the solution to the positive definite total least squares problems and showed that PDTLS-Chol has less computational complexity. In both algorithms, the key point is that since the objective function $f(X)$ is strictly convex on the cone of the symmetric and positive definite matrices, the solution to (\ref{2}) is the symmetric and positive definite matrix $X$ satisfying the first order optimality conditions $\nabla f(X)=0$. It was shown that computing such a matrix is possible using the Cholesky or spectral decomposition of $D^TD$. Here, we intend to make use of PDTLS-Chol to compute a symmetric and positive definite solution to some nonlinear 
equations. The general idea is to propose a linear approximation of the nonlinear equation and solve the corresponding linear problem
 in every iteration. To find a proper linear approximation, we define a suitable change of variables. We also make use of the iterative formula 
$Y_{n+1}=Y_n(2I-XY_n)$, as the Newton's iteration, to converge to the solution of $X-Y^{-1}=0$ which is $X^{-1}$; e.g. see, \cite{22,23}. In a total formulation, both the coefficient and the right hand side matrices are assumed to contain error. Hence, an error is also supposed for the inverse term in the linear subproblems and it seems to be a proper idea to approximate these terms by use of iterative formulas.\\

Therefore, in each iteration of our proposed algorithm for solving a nonlinear matrix equation, a symmetric and positive 
definite solution to the linear approximation of the nonlinear equation is computed using PDTLS-Chol. The process is terminated after satisfaction of a proper
stopping criterion.\\

In Section 3, we discuss solving the nonlinear equation $X+{\sum}_{i=1}^mA_i^Tf_i(X)A_i=Q,$ for the specified three cases mentioned above. In the remainder of our work, by solving a nonlinear equation, we mean finding its symmetric and positive definite solution.
\section{Solving the Nonlinear Equations}
\subsection{Case 1: $m=1, f_1(X)=X^{-1}.$}
Here, the goal is to develop an algorithm to solve the nonlinear matrix equation 
 \begin{equation}\label{8}
X+A^TX^{-1}A=Q,
 \end{equation}
with $A, Q \in {\mathbb{R}}^{n \times n}$. Assuming $Q$ to be the $n \times n$ identity matrix, $I$, the matrix equation 
\[X+A^TX^{-1}A=I,\]
is at hand. This equation arises in different contexts including
analysis of ladder networks, dynamic programming, control theory, stochastic filtering and statistics; e.g., see \cite{2,8,3,4}.\\

Letting $Y=X^{-1}$, (\ref{8}) becomes
\[X+A^TYA=Q.\]
We are to make use of the iterative formula
\[Y_{k+1}=Y_k(2I-XY_k)\]
for the $Y_k$ converging to $X^{-1}$. Thus, to solve (\ref{8}), we define the sequences $Y_{k+1}$ and $X_{k+1}$ by
 \begin{subequations}
 \begin{equation}\label{9}
Y_{k+1}=Y_k(2I-X_kY_k),
 \end{equation}
 \begin{equation}\label{10}
X_{k+1}\simeq Q-A^TY_{k+1}A,
 \end{equation}
 \end{subequations}
starting with arbitrary symmetric and positive definite points $X_0, Y_0 \in {\mathbb{R}}^{n \times n}$. Hence, in each iteration of our proposed algorithm for solving (\ref{8}),
after computing $Y_{k+1}$ from (\ref{9}), we perform PDTLS-Chol for $D=I$ and $T=Q-A^TY_{k+1}A$ to compute $X_{k+1}$. In the remainder of our work, by $X=\textrm{PDTLS-Chol}(D,T)$, we mean that $X$ is computed by applying PDTLS-Chol for the input arguments $D$ and $T$. The advantage of this method, as compared to simply letting $X_{k+1}=Q-A^TY_{k+1}A$, is that $X_{k+1}$ remains positive definite in all iterations. A proper
stopping criterion here would be 
\[E=\|X_{k+1}+A^TY_{k+1}A-Q\|\leq \delta+\epsilon \|X_{k+1}\|,\]
where $\delta$ is close to the machine
(or user's) zero, and $\epsilon$ is close to the unit round-off error.\\

Next, we outline the steps of our proposed algorithm for solving (\ref{8}).\\

\begin{algorithm}[H]
\caption{ Solving the Nonlinear Matrix Equation $X+A^TX^{-1}A=Q$: Nonlinear1.}
\label{alg1}
\begin{algorithmic}[1]
\Procedure {Nonlinear1}{$A$, $\delta$, $\epsilon$}
\State Choose the arbitrary symmetric and positive definite matrices $X, Y \in {\mathbb{R}}^{n \times n}$.
\Repeat
\State Let 
\[Y=Y(2I-XY),\]
\Statex $\hspace{1.1cm}$and
\[X=\textrm{PDTLS-Chol}(I,Q-A^TYA).\]
\State Compute $E=\|X+A^TYA-Q\|$.
\Until{$E\leq \delta+\epsilon \|X\|$.}
\EndProcedure
\end{algorithmic}
\end{algorithm}

\subsection{Case 2: $m=1, f_1(X)=-X^{-2}.$}
Here, we consider solving the nonlinear matrix equation
\begin{equation}\label{11}
X-A^TX^{-2}A=Q,
\end{equation}
where $A \in {\mathbb{R}}^{n \times n}$. This equation arises in solving partial differential equations; e.g., see \cite{5}. As before, defining $Y=X^{-1}$, we get
\[Y^{-1}-A^TY^2A=Q.\]
Hence, the iterative equation
\begin{equation}\label{12}
Y_{k+1}^{-1}-A^TY_{k}^2A=Q
\end{equation}
needs to be solved. We make use of the formula 
\begin{equation}\label{13}
X_{k+1}=X_k(2I-Y_{k+1}X_k),
\end{equation}
converging to $Y_{k+1}^{-1}$. Substituting $X_{k+1}$ in (\ref{12}), we get
\[X_k(2I-Y_{k+1}X_k)-A^TY_{k}^2A=Q,\]
or equivalently,
\[2I-X_kY_{k+1}-A^TY_{k}^2A{X_k}^{-1}-Q{X_k}^{-1}=0.\]
Hence, $Y_{k+1}$ can be computed using 
\begin{equation}\label{14}
Y_{k+1}=\textrm{PDTLS-Chol}(X_k,2I-A^TY_{k}^2A{X_k}^{-1}-Q{X_k}^{-1}).
\end{equation}
Thus, in every iteration of our proposed algorithm, starting from arbitrary symmetric and positive definite matrices $Y_0, X_0\in {\mathbb{R}}^{n\times n}$, we compute $Y_{k+1}$ from (\ref{14}) and $X_{k+1}$ from (\ref{13}). A proper stopping criterion would be 
\[\|X_{k+1}-A^TY_{k+1}^2A-Q\|<\delta+\epsilon \|X_{k+1}\|,\]
with $\delta$ and $\epsilon$ as defined in Case 1. Now, $X_{k+1}$ gives an approximate solution of (\ref{11}). The described steps for solving (\ref{11}) are summerized in the following algorithm.

\begin{algorithm}[H]
\caption{ Solving the Nonlinear Matrix Equation $X+A^TX^{-2}A=Q$: Nonlinear2.}
\label{alg2}
\begin{algorithmic}[1]
\Procedure {Nonlinear2}{$A$, $\delta$}
\State Choose arbitrary symmetric and positive definite matrices $X, Y \in {\mathbb{R}}^{n \times n}$.
\Repeat
\State Let 
\[Y=\textrm{PDTLS-Chol}(X,2I-A^TY^2AX^{-1}-QX^{-1}).\]
\State Let
\[X=X(2I-YX).\]
\State Compute $E=\|X-A^TY^2A-Q\|$.
\Until{$E\leq \delta+\epsilon \|X\|$.}
\EndProcedure
\end{algorithmic}
\end{algorithm}

\subsection{Case 3: $m=2, f_1(X)=X^{-t_1}, f2_(X)=X^{-t_2}.$} The nonlinear matrix equation 
\begin{equation}\label{15}
X^s+A_1^TX^{-t_1}A_1+A_2^TX^{-t_2}A_2=Q
\end{equation}
has applications in different areas such as control theory and dynamical programming; e.g., see \cite{6}. To solve (\ref{15}), we make use of the same change of variables as given in Section 3.2, $Y=X^{-1}$. Substituting $Y$ in (\ref{15}), we get
\[X^s+A_1^TY^{t_1}A_1+{A_2}^TY^{t_2}A_2=Q.\]
Thus, the iterative equation 
\begin{equation}
X_{k+1}^s+A_1^TY_k^{t_1}A_1+A_2^TY_k^{t_2}A_2=Q
\end{equation}
is generated, which is equivalent to
\begin{eqnarray}
U=\textrm{PDTLS-Chol}(I,Q-A_1^TY_k^{t_1}A_1-A_2^TY_k^{t_2}A_2),\label{16}\\
X_{k+1}=U^{\frac{1}{s}}.\nonumber
\end{eqnarray}
To update $Y_k$ to $Y_{k+1}$, one iteration of the formula,
\begin{equation}\label{17}
Y_{k+1}=Y_k(2I-X_{k+1}Y_k),
\end{equation}
is applied. Thus, in each iteration of our proposed algorithm for solving (\ref{15}), starting from an arbitrary symmetric and positive definite $n\times n$ matrix $Y_0$, we compute $X_{k+1}$ using (\ref{16}) and then apply (\ref{17}) to compute $Y_{k+1}$. Instead of starting from an arbitrary matrix $Y_0$, as suggested in \cite{6}, $Y_0=(\frac{\gamma+1}{2\gamma})Q^{-\frac{1}{s}}$ is a suitable starting point. The stopping condition can be set to $E=\|X_{k+1}^s+A_1^TY_k^{t_1}A_1+A_2^TY_k^{t_2}A_2-Q\|<\delta+\epsilon \|X_{k+1}^s\|$ with $\delta$ and $\epsilon$ as before. We now outline our proposed algorithm for solving (\ref{15}).
\begin{algorithm}[H]
\caption{ Solving the Nonlinear Matrix Equation $X^s+A_1^TX^{-t_1}A_1+A_2^TX^{-t_2}A_2=Q$: Nonlinear3.}
\label{alg3}
\begin{algorithmic}[1]
\Procedure {Nonlinear3}{$A_1$, $A_2$, $Q$, $s$, $t_1$, $t_2$, $\delta$}
\State Choose arbitrary symmetric and positive definite matrix $Y \in {\mathbb{R}}^{n \times n}$.
\Repeat
\State Let 
\[U=\textrm{PDTLS-Chol}(I,Q-A_1^TY^{t_1}A_1-A_2^TY^{t_2}A_2)\]
\Statex $\hspace{1.1cm}$and
\[X=U^{\frac{1}{s}}.\]
\State Compute $Y=Y(2I-XY),$ and compute $E=\|U+A_1^TY^{t_1}A_1+$
\Statex $\hspace{1.1cm}A_2^TY^{t_2}A_2-Q\|$.
\Until{$E\leq \delta+\epsilon \|U\|$.}
\EndProcedure
\end{algorithmic}
\end{algorithm}

\paragraph{Note 1}
In \cite{6}, an iterative algorithm was proposed for solving (\ref{15}). An advantage of Algorithm 3 is that in all iterations, $X_k$ remains positive definite because of using PDTLS-Chol for updating $X_k$. 

\paragraph{Note 2}
Considering Case 3 with $m>2$ results in the nonlinear matrix equation
\begin{equation}\label{18}
X^s+{\sum}_{i=1}^mA_i^TX^{-t_i}A_i=Q.
\end{equation}
The nonlinear equation (\ref{18}) arises in the same contexts as (\ref{15}) including control theory and dynamical programming; e.g., see \cite{7}. A similar procedure to Algorithm 3 can be applied to solve (\ref{18}). In each iteration, it is appropriate to let
\[U=\textrm{PDTLS-Chol}(I,Q-{\sum}_{i=1}^mA_i^TY^{t_i}A_i),\]
\[X=U^{\frac{1}{s}},\]
and
\[Y=Y(2I-XY).\] 
The process would be terminated when $E=\|U+{\sum}_{i=1}^mA_i^TY^{t_i}A_i-Q\|<\delta+\epsilon \|U\|$.\\

Next, in Section 4, we discuss the necessary and sufficient conditions for the existence of solution for the three considered cases above.
\section{Existence of Solution}
In \cite{6}, the necessary and sufficient conditions for the exitence of a symmetric and positive definite solution to Case 3 were disscussed. Here, we first recall the conditions in the following theorem and then make use of the theorem for special choices of parameters to provide the necessary and sufficient conditions for the existence of a symmetric and positive definite solution for Case 1. Finally, we point out a theorem from \cite{17} about the sufficient conditions for the existence of a positive definite solution for Case 2.
\begin{theorem}\label{th1}
Case 3 has a unique symmetric and positive definite solution if and only if $A_1$ and $A_2$ can be factored as $A_1={(LL^T)}^{\frac{t_1}{2s}}N_1$ and $A_2={(LL^T)}^{\frac{t_2}{2s}}N_2$, so that the matrix $\left(\begin{array}{c}
  LQ^{-\frac{1}{2}} \\
  N_1Q^{-\frac{1}{2}} \\
  N_2Q^{-\frac{1}{2}}
\end{array}\right)$ has orthogonal columns.
\end{theorem}
\begin{proof}
See \cite{6}.$\hspace{1cm}\square$
\end{proof}
\begin{theorem}\label{th2}
Case 1 has a unique symmetric and positive definite solution if and only if $A$ can be factored as $A={(LL^T)}^{\frac{1}{2}}N$, with $Q^{-\frac{1}{2}}L^TLQ^{-\frac{1}{2}}+Q^{-\frac{1}{2}}N^TNQ^{-\frac{1}{2}}$ being a diagonal matrix.
\end{theorem}
\begin{proof}
Making use of Theorem \ref{th1} for $t_1=1$, $t_2=0$, $A_2=0$ and $s=1$, it can be concluded that Case 1 has a symmetric and positive definite solution if and only if $A$ can be factored as $A={(LL^T)}^{\frac{1}{2}}N$ such that $\left(\begin{array}{c}
  LQ^{-\frac{1}{2}} \\
  NQ^{-\frac{1}{2}}
\end{array}\right)$ has orthogonal columns, or equivalently $Q^{-\frac{1}{2}}L^TLQ^{-\frac{1}{2}}+Q^{-\frac{1}{2}}N^TNQ^{-\frac{1}{2}}$ is a diagonal matrix.$\hspace{1cm}\square$ \\
\end{proof}
In the following theorem, sufficient conditions for existence of a symmetric and positive definite solution to Case 2 are given. 
\begin{theorem}\label{th3}
If there exists an $\alpha >2$ such that
\begin{subequations}
\begin{equation}\label{24}
AA^T>{\alpha}^2(\alpha-1)I,
\end{equation}
\begin{equation}\label{25}
\sqrt{\frac{AA^T}{\alpha-1}}-\frac{1}{{\alpha}^2}AA^T<I,
\end{equation}
\begin{equation}\label{26}
\frac{{\|A\|}^2}{2\alpha{(\alpha -1)}^2}<1,
\end{equation}
\end{subequations}
then, Case 2 has a symmetric and positive definite solution.
\end{theorem}
\begin{proof}
See \cite{17}.$\hspace{1cm}\square$
\end{proof}
\begin{theorem}\label{th4}
If all of the singular values of $A$, ${\sigma}_i$, for $i=1, \cdots, n$, satisfy $\alpha\sqrt{\alpha-1}<{\sigma}_i<\sqrt{2\alpha}(\alpha-1)$, for an $\alpha>2$, then the conditions (\ref{24})-(\ref{26}) are satisfied. Thus, Case 2 has a symmetric and positive definite solution.
\end{theorem}
\begin{proof}
The eigenvalues of $AA^T$ are equal to ${\lambda}_i={\sigma}_i^2$, $i=1, \cdots, n$. Since $\alpha\sqrt{\alpha-1}<{\sigma}_i$, we have ${\alpha}^2(\alpha-1)<{\lambda}_i$ and (\ref{24}) is satisfied. Let $U=\sqrt{AA^T}$. From (\ref{24}), we get
\[U>({\alpha}\sqrt{(\alpha-1)})I,\]
and
\[\sqrt{(\alpha-1)}U-\frac{{\alpha}^2}{2}I>\frac{\alpha}{2}(\alpha-2)I.\]
Thus, we have
\begin{eqnarray}\nonumber
(\sqrt{(\alpha-1)}U-\frac{{\alpha}^2}{2}I)^2&>&(\frac{{\alpha}^4}{4}-{\alpha}^3+{\alpha}^2)I\nonumber\\
&=&\frac{{\alpha}^2}{4}{(\alpha-2)}^2I.\label{28}
\end{eqnarray}
Hence, (\ref{28}) results in
\[{\alpha}^2U-\sqrt{(\alpha-1)}U^2<\sqrt{(\alpha-1)}{\alpha}^2I,\]
or equivalently,
\[\frac{U}{\sqrt{(\alpha-1)}}-\frac{1}{{\alpha}^2}U^2<I.\]
Now, substituting $U$ with $\sqrt{AA^T}$ gives (\ref{25}). Finally, we show that under the assumption ${\sigma}_i<\sqrt{2\alpha}(\alpha-1)$, (\ref{26}) holds. It is sufficient to substitute ${\|A\|}^2$ with $\max ({\lambda}_i)$. Since $\max ({\lambda}_i)<2\alpha{(\alpha-1)}^2$, we have ${\|A\|}^2<2\alpha{(\alpha-1)}^2$.$\hspace{1cm}\square$\\
\end{proof}

Consequently, to generate a test problem in Case 1 having a symmetric and positive definite solution, it is sufficient to choose $A$ with singular values satisfying $\alpha\sqrt{(\alpha-1)}<{\sigma}_i<\sqrt{2\alpha}(\alpha-1)$, for an $\alpha>2$. In the next section, we will use the above results to generate test problems with symmetric and positive definite solutions.
\section{Numerical Results}
We made use of MATLAB 2014a in a Windows 7 machine with a 2.4 GHz CPU and a 6 GB RAM to implement our proposed algorithms and other methods. We then applied the programs on some existing test problems as well as newly generated ones. The numerical results corresponding to cases 1, 2 and 3 are respectively reported in sections 5.1, 5.2 and 5.3.\\

In Section 5.1, first a table is provided in which the selected values for the matrices $A$ and $Q$ are reported. We then report the computing time and the resulting error for solving the nonlinear equation (\ref{8}) using our proposed method and an existing method due to Zhou \cite{8}. In one of these examples, a complex value for the matrix $A$ has been chosen to affirm that our proposed method is applicable to both real and complex problems.\\

In Section 5.2, the computing times and the resulting errors in solving the nonlinear equation (\ref{11}), for almost the same $A$ and $Q$ matrices as given in Section 5.1, are reported using our proposed method and the method discussed by Ivanov \cite{5}. \\

In Section 5.3, we first provide three tables to present the values of $A_1$, $A_2$, $Q$, $s$, $t_1$ and $t_2$ for our test problems. We then report the computing times and the error values for solving the nonlinear equation (\ref{15}) using our proposed method and the method described by Liu \cite{6}. \\

Furthermore, in each section, we generate 100 random test problems satisfying the sufficient conditions for existence of a symmetric and positive definite solution discussed in Section 4. Representing the Dolan-Mor\'{e} time and error profiles, we confirm the efficiency of our proposed algrithms for solving the random test problems. \\

To generate these test problems for cases 1, 2 and 3, we use the results of theorems 2, 4 and 1. respectively. For Case 1, assuming $A={(LL^T)}^{\frac{1}{2}}N$ and $Q=I$, the matrix $\left(\begin{array}{c}
 L \\
  N
\end{array}\right)$ is needed to have orthogonal columns. Thus, to generate a test problem for Case 1 with a symmetric and positive definite solution, it is sufficient to set\\\\
\textbf{Pseudocode 1}
\begin{verbatim}
Q1=qr(rand(2*n));
Q2=Q1(:,1:n);
L=Q2(1:n,:);
N=Q2(n+1:2*n,:);
A=(L*L')^(0.5)*N;
\end{verbatim}
in MATLAB. Similarly, for Case 3, to generate a test problem having a symmetric and positive definite solution, we set \\\\
\textbf{Pseudocode 2}
\begin{verbatim}
Q1=qr(rand(3*n));
Q2=Q1(:,1:n);
L=Q2(1:n,:);
N1=Q2(n+1:2*n);
N2=Q2(2*n+1:3*n);
A1=(L*L')^(t1/(2*s))*N1;
A2=(L*L')^(t2/(2*s))*N2;
\end{verbatim}
Finally, generating a test problem for Case 2 with a symmetric and positive definite solution is possible using the pseudocode below for an arbitrary $\alpha>2$:\\\\
\textbf{Pseudocode 3}
\begin{verbatim}
s1=alpha(sqrt(alpha-1))
s2=sqrt(2alpha)(alpha-1)
d=(s2-s1)*rand(n,1)+s1;
D=diag(d);
U=qr(rand(n));
V=qr(rand(n));
A=U*D*V';
\end{verbatim}
\paragraph{Note} In pseudocodes 1, 2 and 3, the QR factorizations of $2n\times 2n$, $3n\times 3n$ and $n\times n$ matrices need to be computed. Thus, the computing complexity of Pseudocode 3 is lower than the others and as seen in numerical results, it is possible to generate larger test problems for Case 2 without encountering a memory problem.
\subsection{Results for solving (\ref{8})}
In Table \ref{t1}, we see four test examples, assuming $Q=I$ with different values for $A$, to test Nonlinear1 for solving (\ref{8}).
\newpage
\begin{table}[!ht]
\caption{The values of $A$ in test problems.}
\label{t1}
\begin{center}\footnotesize
\begin{tabular}{|c|c|}
    \hline
     &  \\
   Example 1 & $A=\left(
                                               \begin{array}{cccc}
                                                 0.0955 & 0.0797 &  0.0848 & 0.0575\\
                                                0.0920  & 0.0114  & 0.0583 &  0.0010\\
                                                0.0385  & 0.0159  & 0.0586  & 0.0809\\
                                                   0.0163  & 0.0356  & 0.0926  & 0.0609
                                               \end{array}
                                             \right)$\\
                                              &  \\
                                               \hline
                                                &  \\
                                               Example 2 & $A=\left(
                                               \begin{array}{cccc}
                                                 0.8862 & 0.8978 & 0.8194 & 0.4279\\
                                                 0.9311 &  0.5934 & 0.5319 & 0.9661\\
                                                 0.1908 & 0.5038 & 0.2021 & 0.6201 \\
                                                 0.2586 &  0.6128 & 0.4539 & 0.6954
                                               \end{array}
                                             \right)$ \\                              
                                              &  \\
\hline
 &  \\
Example 3 & $A=A_1+iA_2$  \\
& \\
 \hline
  &  \\
Example 4 & $A=\left(
                                               \begin{array}{cccccc}
       0.0450 & 0.0440 & 0.0900 & 0.0660 & 0.0470 &   0.0060\\
       0.0810 & 0.0680 & 0.0550 & 0.0700 & 0.0460 &   0.0140\\
       0.0930 & 0.0470 & 0.0750 & 0.0920 & 0.0810 &  0.0170 \\
       0.0670 & 0.0950 & 0.0120 & 0.0660 & 0.0820 & 0.0630\\
       0.0370 & 0.0350 & 0.0450 & 0.0690 & 0.0190 &  0.0030\\
       0.0410 & 0.0340 & 0.0070 & 0.0850 & 0.0030 & 0.0470
                                               \end{array}
                                             \right)$\\
                                              &  \\
    \hline
\end{tabular}
\end{center}
\end{table}

\begin{center}

$A_1=\left(\begin{array}{cccccc}
0.0320  &  0.0540  & 0.0220 & 0.0370  & 0.0190  & 0.0860\\
0.0120  & 0.0650 & 0.0110 & 0.0760   & 0.0140 & 0.0480 \\
0.0940  &  0.0540 &  0.0110 & 0.0630  & 0.0700  &  0.0390 \\
0.0650  & 0.0720 & 0.0060 & 0.0770  & 0.0090  &  0.0670 \\
0.0480  & 0.0520  & 0.0400 & 0.0930 & 0.0530  & 0.0740 \\
0.0640   & 0.0990 &  0.0450  & 0.0970  & 0.0530 & 0.0520
                                               \end{array}
                                             \right)$
                                             $\vspace{1cm}$
\end{center}

                                             \begin{center}
                                             
                                             $A_2=\left(\begin{array}{cccccc}
0.0350 &  0.0240 & 0.0680 & 0.0270   & 0.0770 & 0.0790\\
0.0150 & 0.0440 &  0.0700 & 0.0200   & 0.0400 & 0.0950\\
 0.0590 &  0.0690 &  0.0440 &  0.0820  & 0.0810 &  0.0330\\
0.0260 & 0.0360 & 0.0020 &  0.0430  & 0.0760 &  0.0670\\
0.0040  &  0.0740 & 0.0330  & 0.0890  &  0.0380 &  0.0440\\
 0.0750  &  0.0390 &  0.0420  & 0.0390   &  0.0220  & 0.0830
                                               \end{array}
                                             \right)$
                                             \end{center}
\normalsize
In Table \ref{t2}, the computed solutions to (\ref{8}) are reported for the considered examples.

\begin{table}[!ht]
\caption{Computed positive definite solutions of (\ref{8}) using Nonlinear1 for examples 1 - 4 as reported in Table \ref{t1}.}
\label{t2}
\begin{center}\footnotesize
\begin{tabular}{|c|c|}
    \hline
     &  \\
   Example 1 & $X=\left(
                                               \begin{array}{cccc}
    1.0009  &  0.0007  &  0.0012  &  0.0009\\
    0.0007  &  1.0005  &  0.0009  & 0.0007\\
    0.0012  &  0.0009  &  1.0016  &  0.0013\\
    0.0009  &  0.0007  &  0.0013  &  1.0010
                                               \end{array}
                                             \right)$\\
                                              &  \\
                                               \hline
                                                &  \\
                                               Example 2 & $X=\left(
                                               \begin{array}{cccc}
                                                 1.9517 &  1.1863  & 0.9614 & 1.0964\\
                                                 1.1863 &  2.3866  & 1.1153 & 1.2316\\
                                                 0.9614 &  1.1153  & 1.9041 & 0.9967 \\
                                                 1.0964 &  1.2316  & 0.9967 & 0.6954
                                               \end{array}
                                             \right)$
                                             \\                              
                                              &  \\
\hline
 &  \\
Example 3 & $X=X_1+iX_2$  \\
& \\
 \hline
  &  \\
Example 4 &  $X=\left(
                                               \begin{array}{cccccc}
       1.2245 & 0.1924 & 0.1858 & 0.2775 & 0.1651 &  0.0823\\
       0.1924 & 1.1642 & 0.1606 & 0.2382 & 0.1422 &  0.0699\\
       0.1858 & 0.1606 & 1.1569 & 0.2301 & 0.1380 &  0.0679 \\
       0.2775 & 0.2382 & 0.2301 & 1.3441 & 0.2027 & 0.1016\\
       0.1651 & 0.1422 & 0.1380 & 0.2027 & 1.1221 &  0.0596\\
       0.0823 & 0.0699 & 0.0679 & 0.1016 & 0.0596 & 1.0295
                                               \end{array}
                                             \right)$ \\
                                              &  \\
    \hline
\end{tabular}
\end{center}
\end{table}

\begin{center}

$X_1=\left(\begin{array}{cccccc}
1.4837 &  0.5838 &  0.3177 &  0.6582  & 0.4726  & 0.6214 \\
   0.5838  & 1.7355  &  0.4195 &  0.8271  & 0.6059 &  0.7910 \\
   0.3177 &  0.4195 & 1.2801 &  0.4660  & 0.3809 &  0.4745\\
   0.6582 &  0.8271  &  0.4660   & 1.9316   & 0.6778 & 0.8879 \\
   0.4726 & 0.6059  & 0.3809  & 0.6778 & 1.5422  & 0.6681 \\
   0.6214 &  0.7910 &  0.4745  & 0.8879  & 0.6681 &  1.8703 
                                               \end{array}
                                             \right)$\end{center} 
                                             \begin{center}
                                             $X_2=\left(\begin{array}{cccccc}
0.0000 &  0.0547 &  0.1716 &  0.0498  & 0.1856  &  0.1349\\
   - 0.0547  & 0.0000 &  0.1793 &  - 0.0210   &  0.1711 &  0.0981\\
   0.1716 & -0.1793 & 0.0000 &  -0.2160  & - 0.0492 & - 0.1332\\
  - 0.0498 &  0.0210 &   0.2160  & - 0.0000 & 0.2104  &  0.1339\\
  - 0.1856 & - 0.1711 & 0.0492 & - 0.2104 & - 0.0000 & - 0.1071\\
    - 0.1349 &  - 0.0981 &  0.1332 & - 0.1339 & 0.1071 &  0.0000
                                               \end{array}
                                             \right)$
\end{center}         
\normalsize              
Table \ref{t3} represents the error values, $E=\|X+A^TX^{-1}A-Q\|$, the computing times, $T$, and the number of iterations, $n_{It}$, for solving (\ref{8}) in examples 1 through 4 using our proposed algorithm, Nonlinear1, and the method introduced by Zhou \cite{8}, denoted by Zhou's algorithm.
\newpage
\begin{table}[htbp]
\caption{Error values, computing times and number of iterations for Nonlinear1 and Zhou's algorithm.}
\label{t3}
\begin{center}\footnotesize
\begin{tabular}{|c|c|c|c|c|c|c|}
\hline
       {\multirow{2}{*}{Example}} &  \multicolumn{3}{c|}{Nonlinear1}  &    \multicolumn{3}{c|}{Zhou's algorithm}\\
     \cline{2-7}
        &  $E$  &  $T$  &   $n_{It}$ & $E$  &  $T$  &   $n_{It}$  \\
      
    \hline
1 & 5.78E-010 & 1.21E-004 & 5 & 1.64E-009 & 2.87E-003 & 4\\
2 & 1.33E-011 & 7.11E-004 & 4 & 3.26E-010 & 1.13E-003 & 6\\
3 & 2.71E-010 & 2.14E-003 & 9 & 5.61E-009 & 4.31E-002 & 9\\
4 & 2.36E-010 & 1.11E-004 & 3 & 1.27E-009 & 7.87E-004 & 8\\
\hline
\end{tabular}
\end{center}
\end{table}
\normalsize

Considering the results in Table \ref{t3}, it is concluded that our proposed method for solving (\ref{8}) computes a symmetric and positive definite solution faster than Zhou's algorithm. The error values however are almost the same. We also compare these two algorithms in solving 100 random test problems later. These test problems are generated randomly using Pseudocode 1. In these test problems, the size of the matrix $A$ is taken to be $5\times 5$, $10\times 10$, $100 \times 100$ or $1200 \times 1200$. It should be mentioned that our proposed algorithm was capable of computing the solution corresponding to a $1200\times 1200$ matrix while the Zhou's algorithm encountered a memory problem. In figures \ref{f1} and \ref{f2}, the Dolan-Mor\'{e} time and error profiles are represented to confirm the efficiency of our proposed algorithm in solving (\ref{8}), showing more efficiency and lower error values.

\begin{figure}[ht!]
       \includegraphics[width=0.7\textwidth]{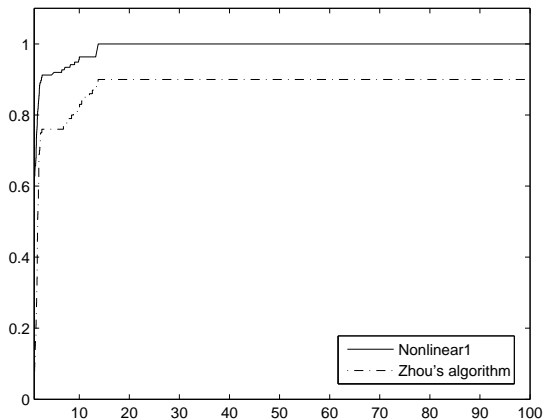}
     \caption{Comparing the computing times by Nonlinear1 and Zhou's algorithms.}
     \label{f1}
\end{figure}
\newpage
\begin{figure}[ht!]
       \includegraphics[width=0.7\textwidth]{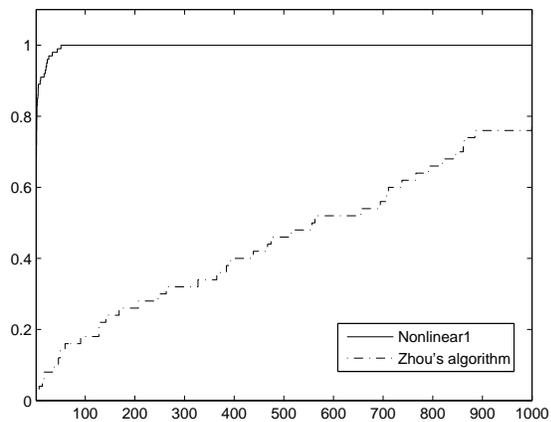}
     \caption{Comparing the error values of Nonlinear1 and Zhou's algorithms.}
     \label{f2}
\end{figure}

\subsection{Results for solving (\ref{11})}
Here, we compare the results of solving (\ref{11}) with use of our proposed algorithm, Nonlinear2, and the algorithm due to Ivanov \cite{17}. The assumed values for $A$ and $Q$ are almost the same as the ones in Section 5.1. The only difference is the value of $A$ in Example 3. The matrix $A$ here is asuumed to be
\begin{center}
$A=\left(\begin{array}{cccc}
-0.1 & -0.1 & 0.02 & 0.08\\
-0.09 & 0.3 & -0.2 & -0.1\\
-0.04 & 0.1 & 0.01 & -0.1\\
-0.08 & -0.06 & -0.1 & -0.2
\end{array}
 \right).$
\end{center}
In Table \ref{t4}, we report the computed solution $X$ of (\ref{11}) using Nonlinear2 for the four examples and with $Q$ being equal to $I$.
\newpage
\begin{table}[!ht]
\caption{Computed positive definite solutions of (\ref{11}) using Nonlinear2 for examples 1 through 4.}
\label{t4}
\begin{center}\footnotesize

\begin{tabular}{|c|c|}
    \hline
     &  \\
   Example 1 & $X=\left(
                                               \begin{array}{cccc}
    0.7138  & 0.4637 &   0.4443 &  -0.3512\\
    0.4637 &   1.3535  &  0.1059  &  0.0952\\
    0.4443  &  0.1059  &  0.9837  &  0.1273\\
   -0.3512  &  0.0952  &  0.1273  &  1.1561
                                               \end{array}
                                             \right)$\\
                                              &  \\
                                               \hline
                                                &  \\
                                               Example 2 & $X=\left(
                                               \begin{array}{cccc}
    0.3837  & 0.0774  & -0.0886 & -0.0423\\
    0.0774 &  0.6145 & 0.2798 &  0.3775\\
   -0.0886  & 0.2798  &  0.7603  & 0.6386\\
   -0.0423  &  0.3775 &  0.6386  &  0.7550
                                               \end{array}
                                             \right)$
                                             \\                              
                                              &  \\
\hline
 &  \\
Example 3 &  $X=\left(
                                               \begin{array}{cccccc}
                                                0.9877 &   0.0125  &  0.0068 &   0.0170\\
    0.0125  &  1.0821 &  -0.0433  & -0.0525\\
    0.0068  & -0.0433  &  1.0540  &  0.0560\\
    0.0170  & -0.0525  &  0.0560 &   1.0621                                               \end{array}
                                             \right)$  \\
& \\
 \hline
  &  \\
Example 4 &  $X=\left(
                                               \begin{array}{cccccc}
 1.1818  &  0.0117  & -0.0053  & -0.0099  &  0.0477 & -0.0952\\
 0.0117  &  0.9994  & -0.0202  & -0.0168 &  0.0000 & -0.0160\\
 -0.0053 & -0.0202  & 0.9526   & -0.0482 & -0.0181 & -0.0341\\
 -0.0099 & -0.0168  & -0.0482  &  0.9554 & -0.0181 & -0.0289\\
  0.0477 & 0.0000   & -0.0181  & -0.0181 &  1.0107 & -0.0345\\
 -0.0952 & -0.0160  & -0.0341  &  -0.0289  & -0.0345 &  1.0267
                                               \end{array}
                                             \right)$ \\
                                              &  \\
    \hline
\end{tabular}
\end{center}
\end{table}

\normalsize
In Table \ref{t5}, to compare our proposed method in solving (\ref{11}) for examples 1 through 4 with the method due to Ivanov \cite{17}, denoted by Ivanov's algorithm, we report error values, $E=\|X-A^TX^{-2}A-Q\|$, the computing times, $T$, and the number of iterations, $n_{It}$.

\begin{table}[htbp]
\caption{Computing times, error values and number of iterations for Nonlinear2 and Ivanov's algorithm.}
\label{t5}
\begin{center}\footnotesize
\begin{tabular}{|c|c|c|c|c|c|c|}
    \hline
       {\multirow{2}{*}{Case}} &  \multicolumn{3}{c|}{Nonlinear 2}  &    \multicolumn{3}{c|}{Ivanov Algorithm}\\
     \cline{2-7}
        &  $E$  &  $T$  &   $n_{It}$ & $E$  &  $T$  &   $n_{It}$ \\
      
    \hline
1 & 8.61E-011 & 1.23E-003 & 2 & 1.64E+000 & 5.32E+000 & 27\\
2 & 6.79E-010 & 9.24E-004 & 3 & 1.09E-009 & 1.11E-002 & 13\\
3 & 1.28E-010 & 2.59E-002 & 55 & 3.77E+000 & 1.14E-001 & 100\\
4$^{\ast}$ & 2.51E-011 & 7.85E-004 & 12 & 5.43E+000 & 2.04E-002 & 28\\
\hline
\end{tabular}
\end{center}
\end{table}
\normalsize
\vspace{0.1cm}
$\ast:$ Chosen from \cite{17}.\\
\newpage
The reported results in Table \ref{t5} show that our proposed algorithm computes the symmetric and positive definite solution to (\ref{11}) faster and with lower error values than Ivanov's algorithm. The Dolan-Mor\'{e} time and error profiles are presented in figures \ref{f3} and \ref{f4} to confirm the efficiency of our proposed algorithm in computing a symmetric and positive definite solution to (\ref{11}) over randomly generated test problems using Pseudocode 3. The size of the matrix $A$ is taken to be $10\times 10$, $100 \times 100$, $1000 \times 1000$ or $3000\times 3000$. Although both algorihms were able to compute the solutions for large matrices, as shown in figures \ref{f3} and \ref{f4}, the computing times and the error values are lower for our proposed algorithm.

\begin{figure}[ht!]
       \includegraphics[width=0.7\textwidth]{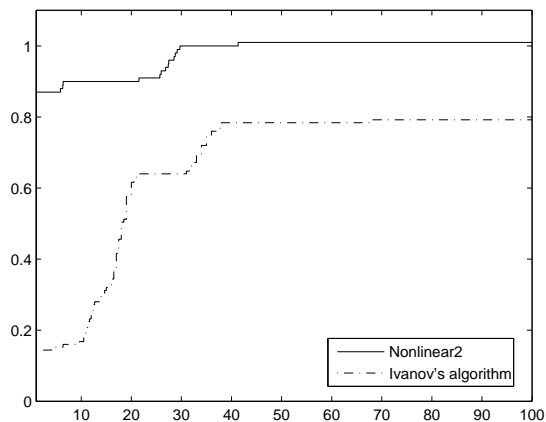}
     \caption{Comparing the computing times by Nonlinear2 and Ivanov's algorithm.}
     \label{f3}
\end{figure}
\newpage
\begin{figure}[ht!]
       \includegraphics[width=0.7\textwidth]{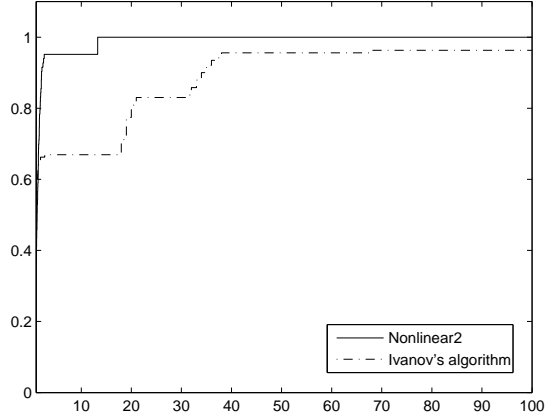}
     \caption{Comparing the error values of Nonlinear2 and Ivanov's algorithm.}
     \label{f4}
\end{figure}
%figure 2

\subsection{Results for solving (\ref{15})}
Two test problems are reported in Table \ref{t6}. The first one is chosen from \cite{6}.\\

\begin{table}[!ht]
\caption{The values of $A$, $B$, $Q$, $s$, $t_1$ and $t_2$ for the test problems.}
\label{t6}
\begin{center}\footnotesize
\begin{tabular}{|c|c|}
    \hline
Example 1$^{\ast}$ & Example 2\\
\hline
$A=\left(
                                               \begin{array}{cccccc}
     2  &   0   &  0   &  1  &   0  &   0\\
     1  &   2 &    0  &   0   &  1   &  0\\
     0   &  0  &   3   &  0  &   1  &   0\\
     1   &  0   &  0   &  2   &  0  &   1\\
     1   &  0   &  1   &  0  &   3  &   0\\
     0   &  1   &  0   &  0  &   1   &  2
                                               \end{array}
                                             \right)$ & $A=\left(
                                               \begin{array}{cccc}
0.5853 & 0\\
0 & 0.5497
                                               \end{array}
                                             \right)$ \\       
                                             \hline
$B=\left(
                                               \begin{array}{cccccc}
   
     2   &  1   &  6  &   0  &   5  &   7\\
     3  &   4   &  7  &   1  &   3  &   0\\
     0  &   9   &  2  &   4  &   7  &   8\\
     8  &   5  &   3  &   0  &   0  &   1\\
     2  &   5  &   0   &  2  &   1  &   7\\
     4   &  0   &  0  &   1   &  4  &   9
                                               \end{array}
                                             \right)$ & $B=\left(
                                               \begin{array}{cccc}
0.9172 & 0\\
0 & 0.2858
                                               \end{array}
                                             \right)$ \\       
                                             \hline
$Q=\left(
                                               \begin{array}{cccccc}

   105 &   66  &  58  & 15  &  41  &  73\\
    66 &  154  &  67  &  50  &  88 &  121\\
    58  &  67 &  109  &  15  &  71  &  61\\
    15  &  50 &   15  &  28  &  37 &   57\\
    41  &  88 &   71  &  37 &  113 &  136\\
    73 &  121 &   61  &  57 &  136 &  250

                                               \end{array}
                                             \right)$ & $Q=\left(
                                               \begin{array}{cccc}
0.3786 & 0\\
0 & 0.3769
                                               \end{array}
                                             \right)$ \\                             
    \hline
$s=5,\hspace{0.2cm} t_1=0.2,\hspace{0.2cm}t_2=0.5$ & $s=2,\hspace{0.2cm} t_1=t_2=0.5$\\
\hline
\end{tabular}
\end{center}
\end{table}
\vspace{0.1cm}
$\ast:$ Chosen from \cite{6}.\\

In Table \ref{t7}, the error values, $E$, the computing times, $T$ and the number of iterations, $n_{It}$, are reported for solving (\ref{15}) for examples 1 and 2 as in Table \ref{t6}, using our proposed method, Nonlinear3, and the method given by Liu \cite{6}, denoted by Liu's algorithm.

\begin{table}[htbp]
\caption{Error values, computing times and number of iterations for Nonlinear3 and Liu's algorithm.}
\label{t7}
\begin{center}\footnotesize
\begin{tabular}{|c|c|c|c|c|c|c|}
    \hline
       {\multirow{2}{*}{Case}} &  \multicolumn{3}{c|}{Nonlinear 3}  &    \multicolumn{3}{c|}{Liu Algorithm}\\
     \cline{2-7}
        &  $E$  &  $T$  &   $n_{It}$ & $E$  &  $T$  &   $n_{It}$  \\
      
    \hline
1 & 1.71E-008 & 5.64E-003 & 3 & 1.12 & 4.41E-004 & 1\\
1 & 1.52E-011 & 8.38E-002 & 8 & 4.33E+097 & 1.31E-002 & 10\\
2 & 3.93E-011 & 3.35E-002 & 29 & 3.93E-011 & 1.50E-002 & 29\\

\hline
\end{tabular}
\end{center}
\end{table}
 As seen in Table \ref{t7}, our proposed method for solving (\ref{15}) computes the solution faster and with lower error values in example 1 and both methods perform exactly the same for the second example. Also, 100 random test problems were generated using Pseudocode 2. Here, the matrices $A_1$ and $A_2$ are taken to be $5\times 5$, $10\times 10$, $100 \times 100$ or $1000\times 1000$. For $1000\times 1000$ matrices $A_1$ and $A_2$, our proposed algorithm could compute the solution while Liu's algorithm encountered a memory problem. The Dolan-Mor\'{e} time and error profiles for these test problems are presented in figures \ref{f5} and \ref{f6}. Considering figures \ref{f5} and \ref{f6}, it can be concluded that our proposed algorithm for solving (\ref{15}) computes the symmetric and positive definite solution faster and with lower error values. 
 
 \begin{figure}[ht!]
       \includegraphics[width=0.7\textwidth]{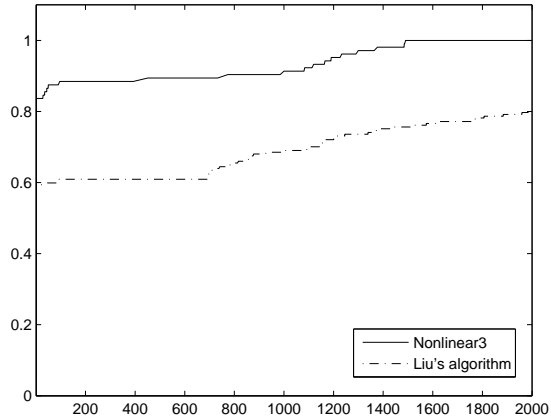}
     \caption{Comparing the computing time for Nonlinear3 and Liu's algorithms.}
     \label{f5}
\end{figure}

\begin{figure}[ht!]
       \includegraphics[width=0.7\textwidth]{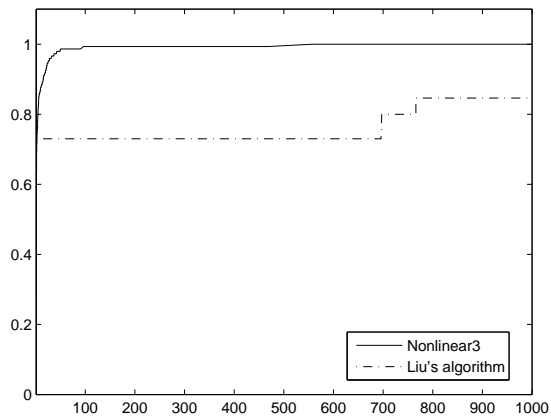}
     \caption{Comparing the error values for Nonlinear3 and Liu's algorithms.}
     \label{f6}
\end{figure}
 
\section*{Concluding Remarks}
Making use of our recently proposed method for solving positive definite total least squares poblems, we presented iterative methods for solving three types of nonlinear matrix equations. We provided linear iterative formulas by use of special convergent formulas and dedfining proper change of variables. To solve the resulting linear peoblems, we used our recently proposed method for solving total positive definite least squares problems (PDTLS-Chol). Compared with other methods, use of PDTLS-Chol for solving the linear problems offers two useful features. First, the solution in all iterations remains positive definite. Second, in a total formulation, as used in PDTLS-Chol, the right hand side matrix of a linear system is also assumed to contain error, and hence approximation of the inverse in the right hand side of the linear equations is not problematic. We outlined three specific algorithms for solving the three types of nonlinear matrix equations having applications in control theory and numerical solutions of partial differential equations. We then experimented with some existing test problems as well as our randomly generated ones for each of the three problem types and reported the corresponding numerical results. Comparing with the existing methods, the reported Dolan-Mor\'{e} profiles confirm the effectiveness of our proposed algorithms in computing a positive definite solution with lower error values and lower computing times.

\section*{Acknowledgements}
The authors thank Research Council of Sharif University of Technology for supporting this work.

\section*{References}

\bibliography{mybibfile}

\begin{thebibliography}{10}
\expandafter\ifx\csname url\endcsname\relax
  \def\url#1{\texttt{#1}}\fi
\expandafter\ifx\csname urlprefix\endcsname\relax\def\urlprefix{URL }\fi
\expandafter\ifx\csname href\endcsname\relax
  \def\href#1#2{#2} \def\path#1{#1}\fi

\bibitem{2}
M.~Li, Y.~Yang, Q.~Li, A new algorithm for solving the equation
  ${X}+{A}^{\ast}{X}^{-1}{A}={I}$, Proc.\ 5th International Conference on
  Computational Sciences and Optimization (CSO, June 2012).

\bibitem{5}
I.~G. Ivanov, S.~M. El-Sayed, Properties of positive definite solution of the
  equation ${X}+{A}^{\ast}{X}^{-2}{A}={I}$, Linear ALgebra Appl 279 (1998)
  303--316.

\bibitem{6}
A.~Liu, G.~Chen, On the hermitian positive definite solutions of the nonlinear
  matrix equation ${X}^s+{A}^{\ast}{X}^{-t_1}{A}+{B}^{\ast}{X}^{-t_2}{B}={Q}$,
  Mathematical Problems in Engineering. $\hspace{0.1cm}$\href
  {http://dx.doi.org/10.1155/2011/163585} {\path{doi:10.1155/2011/163585}}.

\bibitem{19}
S.~M. El-Sayed, A.~C.~M. Ran, On an ietration method for solving a class of
  nonlinear matrix equations, SIAM Math Anal 23(3) (2001) 632--645.

\bibitem{8}
B.~Zhou, G.~B. Cai, J.~Lam, Positive definite solutions of the nonlinear matrix
  equation ${X}+{A}^{T}{{\bar{X}}}^{-1}{A}={I}$, Appl Math Comput 219 (2013)
  7377--7391.

\bibitem{9}
V.~I. Hasanov, Positive definite solutions of the matrix equation
  ${X}+{A}^{T}{{\bar{X}}}^{-q}{A}={Q}$, Linear Algebra Appl 404 (2005)
  166--182.
\newblock \href {http://dx.doi.org/10.1016/j.laa.2005.02.024}
  {\path{doi:10.1016/j.laa.2005.02.024}}.

\bibitem{21}
Q.~Li, P.~Liu, On hermitian positive definite solution of nonlinear matrix
  equation, Proc.\ International Conference on Enginnering Design and
  Manufacturing Informatization (ICSEM, Oct. 2011)\href
  {http://dx.doi.org/10.110/ICSSEM.2011.6081228}
  {\path{doi:10.110/ICSSEM.2011.6081228}}.

\bibitem{10}
X.~Y. Yin, S.~Liu, T.~Li, On positive definite solutions of the matrix equation
  ${X}+{A}^{T}{{\bar{X}}}^{-q}{A}={Q}$, Taiwanese J Math 16~(4) (2012)
  1391--1407.

\bibitem{13}
C.~Guo, Y.~Kuo, W.~Lin, Numerical solution of nonlinear matrix equations
  arising from greens function calculation, J Comput Appl Math 236~(17) (2012)
  4166--4180.

\bibitem{17}
I.~G. Ivanov, V.~I. Hasanov, B.~V. Minchev, On matrix equations ${X}\pm
  {A}^{\ast}{X}^{-2}{A}={I}$, Linear ALgebra Appl 326 (2001) 27--44.

\bibitem{11}
M.~Cheng, S.~Xu, Perturbation analysis of the hermitian positive deﬁnite
  solution of the matrix equation ${X}-{A}^{T}{X}^{-2}{A}={I}$, Linear Algebra
  Appl 394 (2005) 39--51.
\newblock \href {http://dx.doi.org/10.1016/j.laa.2004.05.013}
  {\path{doi:10.1016/j.laa.2004.05.013}}.

\bibitem{14}
C.~Guo, Y.~Kuo, W.~Lin, Two iteration processes for computing positive definite
  solutions of the equation ${X}-{A}^{\ast}{X}^{-n}{A}={Q}$, Comput Math Appl
  236~(17) (2012) 4166--4180.
\newblock \href {http://dx.doi.org/10.1016/S0898-1221(00)00301-1}
  {\path{doi:10.1016/S0898-1221(00)00301-1}}.

\bibitem{12}
J.~Long, X.~Hu, L.~Zhang, On the hermitian positive deﬁnite solution of the
  nonlinear matrix equation
  ${X}+{A}^{\ast}{X}^{-1}{A}+{B}^{\ast}{X}^{-1}{B}={I}$, Bult Brazilian Math
  Society 39~(3) (2008) 371--386.

\bibitem{20}
W.~Pei, G.~Wu, D.~Zhou, Y.~Liu, Some investigation on hermitian positive
  definite solutions of a nonlinear matrix equation, Int J Comput Math 91(5)
  (2014) 872--880.

\bibitem{18}
X.~Duan, A.~Liao, On the existence of hermitian positive definite solution of
  the matrix equations ${X}^s+ {A}^{\ast}{X}^{-t}{A}={Q}$, Linear ALgebra Appl
  429 (2008) 673--687.

\bibitem{15}
F.~Gibson, F.~Liu, X.~Shi, H.~Umoh, Two kinds of nonlinear matrix equations and
  their corresponding matrix sequences, Linear and Multilinear Algebra 52
  (2004) 1--15.
\newblock \href {http://dx.doi.org/10.1080/0308108031000112606}
  {\path{doi:10.1080/0308108031000112606}}.

\bibitem{16}
X.~Duan, Z.~Peng, F.~Duan, Positive definite solution of two kinds of nonlinear
  matrix equations, Survey Math Appl 4 (2009) 179--190.

\bibitem{1}
N.~Bagherpour, N.~Mahdavi-Amiri, Direct methods for solving positive definite
  total least squares problems using orthogonal matrix decompositions,
  \url{http://arxiv.org/pdf/1407.1372v1.pdf}.

\bibitem{22}
A.~Ben-Israel, An iterative method for computing the generalized inverse of an
  arbitrary matrix, Math Comp 19 (1965) 452--455.

\bibitem{23}
A.~Ben-Israel, D.~Cohen, On iterative computation of generalized inverses and
  associated projections, SIAM J Numer Anal 3 (1966) 410--419.

\bibitem{3}
W.~N. {Anderson Jr.}, T.~D. Morely, G.~E. Trap, Positive solution to
  ${X}={A}-{B}{X}^{-1}{B}^{T}$, Linear ALgebra Appl 194 (1993) 91--108.
\newblock \href {http://dx.doi.org/10.1016/0024-3795(90)90005-W}
  {\path{doi:10.1016/0024-3795(90)90005-W}}.

\bibitem{4}
P.~Lancaster, L.~Rodman, Algebraic Riccati Equations, Oxford Science, UK, 1995.

\bibitem{7}
A.~M. Sarhan, N.~El-Shazly, E.~N. Shehata, On the existence of extremal
  positive definite solutions of the nonlinear matrix equation
  ${X}^r+{\sum}_{i=1}^m{A}^{\ast}{X}^{{\delta}_j}{A}={I}$, Math Comput Model
  51~(9--10) (2010) 1107--1117.
\newblock \href {http://dx.doi.org/10.1016/j.mcm.2009.12.021}
  {\path{doi:10.1016/j.mcm.2009.12.021}}.

\end{thebibliography}

\end{document}